\documentclass[11pt]{amsart}
\usepackage{bbm}
\usepackage{amsmath}  
\usepackage{amsthm}     
\usepackage{enumerate} 
\usepackage{physics}
\usepackage{mathtools}
\usepackage{hyperref}
\usepackage[margin=1.0in]{geometry} 
\usepackage{graphicx}
\usepackage{comment}

\newcommand\smallO{
  \mathchoice
    {{\scriptstyle\mathcal{O}}}
    {{\scriptstyle\mathcal{O}}}
    {{\scriptscriptstyle\mathcal{O}}}
    {\scalebox{.7}{$\scriptscriptstyle\mathcal{O}$}}
  }

\newtheorem{prop}{Proposition}

\numberwithin{remark}{section}
\numberwithin{prop}{section}
\numberwithin{thm}{section}
\numberwithin{lemma}{section}
\numberwithin{definition}{section}
\numberwithin{equation}{section}

\title[]{Malicious Experts versus the multiplicative weights algorithm in online prediction}
\author[]{Erhan Bayraktar} 
\address{Department of Mathematics, University of Michigan}
\email{erhan@umich.edu}
\author[]{H. Vincent Poor} 
\address{Department of Electrical Engineering, Princeton University}
\email{poor@princeton.edu}
\author[]{Xin Zhang} 
\address{Department of Mathematics, University of Michigan}
\email{zxmars@umich.edu}
\begin{document}
\maketitle
\begin{abstract}
We consider a prediction problem with two experts and a forecaster. We assume that one of the experts is honest and makes correct prediction with probability $\mu$ at each round. The other one is malicious, who knows true outcomes at each round and makes predictions in order to maximize the loss of the forecaster. Assuming the forecaster adopts the classical multiplicative weights algorithm, we find an upper bound \eqref{eq:upper} for the value function of the malicious expert, and also a lower bound \eqref{eq:lower}. Our results imply that the multiplicative weights algorithm cannot resist the corruption of malicious experts. We also show that an adaptive multiplicative weights algorithm is asymptotically optimal for the forecaster, and hence more resistant to the corruption of malicious experts. 
\end{abstract}

\section{Introduction}
Prediction with expert advice is classical and fundamental in the field of \emph{online learning}, and we refer the reader to \cite{MR2409394} for a nice survey. In this problem, a forecaster makes predictions based on advices of experts so as to minimize his loss, i.e., the cumulative difference between his predictions and true outcomes. A standard performance criterion is the regret: the difference between the loss of the forecaster and the minimum among losses of all experts. The prediction problem is often studied in the so-called adversarial setting and the stochastic setting. In the adversarial setting, the advice of experts is chosen by an adversary so as to maximize the regret of the forecaster, and therefore the problem can be viewed as a zero-sum game between the forecaster and the adversary (see e.g. \cite{MR1265851} \cite{MR3478415} \cite{Drenska2019} \cite{2019arXiv190202368B} \cite{doi:10.1080/03605302.2020.1712418}). In the stochastic setting, the losses of each expert are drawn independent and identically distributed $(i.i.d.)$ over time from a fixed but unknown distribution, and smaller regrets can be achieved compared with the adversarial setting (see e.g. \cite{MR3214784} \cite{10.5555/3157382.3157596} \cite{MR3960937}). 

In this paper, we consider the model in \cite{2020arXiv200100543R} which considers a mix of adversarial and stochastic settings. It is a learning system with two experts and a forecaster. One of the experts is honest, who at each round makes a correct prediction with probability $\mu$. The other one is malicious, who knows the true outcome at each round and makes his predictions so as to maximize the loss of the forecaster. Here we assume that the forecaster adopts the classical multiplicative weights algorithm, and study its resistance to the corruption of the malicious expert. Denote by $V^{\alpha}(N,\rho)$ the expected cumulative loss for the forecaster, where $\alpha$ is the strategy chosen by the malicious expert, $N$ is the fixed time horizon, and $\rho$ is the initial weight of the malicious expert. Instead of regret, we analyze the asymptotic maximal loss  $\lim\limits_{N \to \infty} \max\limits_{\alpha} \frac{V^{\alpha}(N, 1/2)}{N}$.

It was proved in \cite{2020arXiv200100543R} that if the malicious expert is only allowed to adopt offline policies, i.e., to decide whether to tell the true outcome at each round at the beginning of the game, then we have $\lim\limits_{N \to \infty} \max\limits_{\alpha} \frac{V^{\alpha}(N, 1/2)}{N}=1-\mu$. It implies that the extra power of the malicious expert cannot incur extra losses to the forecaster. 

Here we allow the malicious expert to adopt online policies, i.e., at each round, the malicious expert chooses whether to tell the truth based on all the prior histories. To find an upper bound on asymptotic losses, we rescale dynamic programming equations of the problem and obtain a partial differential equation (PDE). Then we prove that the unique solution of this PDE provides us an upper bound
$$\limsup\limits_{N \to \infty} \max\limits_{\alpha} \frac{V^{\alpha}(N, 1/2)}{N} \leq 1-\mu^2.$$
For the lower bound, we design a simple strategy for the malicious expert and prove that 
$$\liminf\limits_{N \to \infty} \max\limits_{\alpha} \frac{V^{\alpha}(N, 1/2)}{N} > 1-\mu,$$
which implies that the malicious expert can incur extra losses to the forecaster when online policies are admissible. To make the forecaster more resistant to the malicious expert, we consider an adaptive multiplicative weights algorithm and prove that it is asymptotically optimal for the forecaster. 

The rest of the paper is organized as follows. In Section~\ref{section:Problem}, we mathematically formulate this problem and develop its dynamic programming equations. In Section~\ref{section:upper}, we show the upper bound of asymptotic losses, and in Section~\ref{section:lower} we find the lower bound. In Section~\ref{section:adaptive}, we consider the malicious expert versus the adaptive multiplicative weights algorithm. In Section~\ref{section:conclusions}, we summarize our results and their implications.

\section{Problem Formulation}\label{section:Problem}
In this section, we introduce the mathematical model as in \cite{2020arXiv200100543R}. Consider a learning system with two experts and a forecaster. For each round $t\in \mathbb{N}_+$, denote the prediction of expert $i \in \{1,2\}$ by $x_t^i \in \{0,1\}$, and the true outcome by $y_t \in \{0,1\}$. 

Suppose that the forecaster adopts the multiplicative weights algorithm. For each round $t \in \mathbb{N}_+$, denote by $p_t^i$ the weight of expert $ i \in \{1,2\}$, $p_t^1+p_t^2=1$. Then the prediction of the forecaster is  $$\hat{y}_t:= \sum_{i=1}^2 p_t^i x_t^i.$$
Given $\epsilon \in (0,1)$, the weights evolve as follows 
\begin{align*}
p_{t+1}^i= \frac{p_t^i \epsilon^{|x^i_t-y_t|}}{p_t^1 \epsilon^{|x^1_t-y_t|}+p_t^2 \epsilon^{|x^2_t-y_t|}}, \quad i=1,2.
\end{align*}

Denote the entire history up to round $t-1$ by $$\mathcal{G}_{t}:=\{p_l^1,p_l^2, x_l^1, x_l^2, y_l: l=1,\dotso t-1\} \cup \{p_t^1,p_t^2 \}.$$
Assume expert $2$ is honest, and at each round $t \in \mathbb{N}_+$ make correct predictions with probability $\mu \in (0,1)$ independently of $\mathcal{G}_t$, i.e., 
\[
x_t^2=
\begin{cases}
y_t   &\text{with probability } \mu, \\
1-y_t &\text{with probability } 1- \mu.
\end{cases}
\]
Expert $1$ is malicious and knows the accuracy $\mu$ of expert $2$ and the outcome $y_t$ at each round. At each stage $t \in \mathbb{N}_+$, based on the information $\mathcal{G}_{t}$, the malicious expert can choose to lie, i.e., make $x_t^1=1-y_t$, or to tell the truth, i.e., make $x_t^1=y_t$. Denote by $\mathcal{A}_{t}$ the space of functions from $\mathcal{G}_{t}$ to $\{T,L\}$, where $T$ (truth) and $L$ (lie) represent $x_t^1=y_t$ and $x_t^1=1-y_t$ respectively. 

At each round $t \in \mathbb{N}_+$, the loss of the forecaster is $l(\hat{y}_t, y_t):=|\hat{y}_t-y_t|$, which is also the gain of the malicious expert. It can be easily verified that 
\begin{equation}\label{eq:loss}
l(\hat{y}_t, y_t)=
\begin{cases}
p_t^1 \quad & \text{if } \alpha_t=L, x_t^2=y_t, \\
1 \quad & \text{if } \alpha_t=L, x_t^2=1-y_t, \\
0 \quad & \text{if } \alpha_t=T, x_t^2=y_t, \\
1-p_t^1 \quad & \text{if } \alpha_t=T, x_t^2=1-y_t.
\end{cases}
\end{equation}
And the evolution of $p_t^1$ is as follows:
\begin{equation}\label{eq:transit}
p_{t+1}^1=
\begin{cases}
g(p_t^1) & \text{if } \alpha_t=L, x_t^2=y_t, \\
g^{(-1)}(p_t^1) & \text{if } \alpha_t=T, x_t^2=1-y_t, \\
p_t^1 & \text{otherwise},
\end{cases}
\end{equation}
where 
$$g(p_t^1)=\frac{1}{1+(1/p_t^1-1)/\epsilon}, \quad g^{(-1)}(p_t^1)=\frac{1}{1+(1/p_t^1-1)\epsilon}.$$ 
For a fixed time horizon $N$, the goal of the malicious expert is to maximize the cumulative loss of the forecaster by choosing a sequence of strategies $\alpha=\{(\alpha_1, \alpha_2, \dotso ): \alpha_t \in \mathcal{A}_t, t \in \mathbb{N}_+ \}$, i.e., solving the optimization problem
\begin{align*}
V(N,\rho):=\max\limits_{\alpha} \mathbb{E}^{\alpha}\left[\sum_{t=1}^N l(\hat{y}_t, y_t)\left| \ p_1^1=\rho \right.\right].
\end{align*}

According to \eqref{eq:loss},
we obtain the expected current loss
\begin{equation}\label{eq:currentloss}
\mathbb{E}^{\alpha_t}\left[l(\hat{y}_t,y_t)| \mathcal{G}_t \right]=
\begin{cases}
(1-\mu+\mu p_t^1) & \text{if } \alpha_t=L, \\
(1-\mu)(1-p_t^1) & \text{if } \alpha_t= T. 
\end{cases}
\end{equation}
In combination with \eqref{eq:transit}, we get dynamic programming equations
\begin{align}\label{eq:dpp}
V(t+1,\rho) =& \max\{(1-\mu+\mu\rho)+\mu V(t, g(\rho))+(1-\mu)V(t,\rho),  \\
& (1-\mu)(1-\rho)+(1-\mu) V\left(t, g^{(-1)}(\rho)\right)+\mu V(t,\rho)       \}, \notag
\end{align}
together with initial conditions $V(0,\rho)=0$.

\section{Upper bound on the Value function}\label{section:upper}
In this section, we properly rescale the \eqref{eq:dpp} and obtain a PDE \eqref{eq:HJBc}. We explicitly solve this equation, and show that its solution \eqref{eq:sol} provides an upper bound on $\limsup\limits_{N \to \infty} \frac{V(N,1/2)}{N}.$ 
\subsection{Limiting PDE}
To appropriately rescale \eqref{eq:dpp} and follow the formulation of \cite{MR3206981}, we change the variable $$x=\frac{\ln(1/\rho-1)}{\ln(1/\epsilon)}, \quad \rho = \frac{1}{1+(1/\epsilon)^x},$$
and define 
\begin{align*}
\tilde{V}(t,x):=-V\left(t, \frac{1}{1+(1/\epsilon)^x}\right).
\end{align*}
Then \eqref{eq:dpp} becomes 
\begin{align}\label{eq:cdpp}
\tilde{V}(t+1,x) =& \min\left\{-\left(1-\mu+\frac{\mu}{1+(1/\epsilon)^x}\right) +\mu \tilde{V}(t, x+1)+(1-\mu)\tilde{V}(t,x),  \right. \\
& \left.-(1-\mu)\left(1-\frac{1}{1+(1/\epsilon)^x}\right)+(1-\mu) \tilde{V}\left(t, x-1\right)+\mu \tilde{V}(t,x)       \right\}. \notag
\end{align}

Define scaled value functions via the equation $\frac{\tilde{V}^{\delta}(\delta t, \delta x)}{\delta}=\tilde{V}(t,x)$. Substituting in \eqref{eq:cdpp}, we obtain that 
\begin{align}\label{eq:sdpp}
\tilde{V}^{\delta}(t+\delta,x)= & \min \left\{   -\delta\left(1-\mu+\frac{\mu}{1+(1/\epsilon)^{x/\delta}} \right)  +\mu \tilde{V}^{\delta}(t, x+\delta) +(1-\mu) \tilde{V}^{\delta}(t,x) \right.  \\
     &  \left.  -\delta(1-\mu)\left(1-\frac{1}{1+(1/\epsilon)^{x/\delta}}\right)+ (1-\mu) \tilde{V}^{\delta}\left(t, x-\delta ) \right)+\mu \tilde{V}^{\delta}(t,x) \right\}. \notag
\end{align} 
Taking $\delta$ to $0$ in \eqref{eq:sdpp}, we obtain a first order PDE 
\begin{align}\label{eq:pde}
0=v_t(t,x)+& \max \left\{   1-\mu+\mu s(x)- \mu v_{x}(t,x), \right. \\
&\left. (1-\mu)(1-s(x))+(1-\mu) v_x(t,x)   \right\}, \notag 
\end{align}
where $v(0,x)=0$, and
\[s(x)=
\begin{cases}
0, \quad \text{if } x>0, \\
1, \quad \text{if } x<0.
\end{cases}
\]
Define $\Omega_1=\{x>0\}, \Omega_2=\{x<0\}, \mathcal{H}=\{x=0\}$, and Hamiltonians
\begin{align*}
H_1(x,p)&= \max \{1-\mu-\mu p, 1-\mu+(1-\mu) p \}, \quad x \in \bar{\Omega}_1, \\
H_2(x,p)&= \max \{1-\mu p, (1-\mu) p \}, \quad x \in \bar{\Omega}_2. \\
\end{align*}
Then \eqref{eq:pde} becomes 
\begin{align}\label{eq:HJB}
v_t+H_i(x,v_x)=0 \quad \text{ for }  x \in \Omega_i,\ i=1,2. 
\end{align}
Following Ishii's definition of viscosity solutions to discontinuous Hamiltonians, we complement \eqref{eq:HJB} by 
\begin{align*}
\min \{v_t+H_1(x,v_x), v_t+H_2(x,v_x) \} \leq 0 \quad \text{ for }  x \in \mathcal{H}, \\
\max \{v_t+H_1(x,v_x), v_t+H_2(x,v_x) \} \geq 0 \quad \text{ for } x \in \mathcal{H}, 
\end{align*}
where $\min$ and $\max$ should be understood in the sense of viscosity solutions. 

Solving \eqref{eq:HJB} by the method of characteristics and assuming that the value function is differentiable with respect to $x$ on $\mathcal{H}$, we  conjecture the solution
\begin{align}\label{eq:sol}
 v(t,x)=
\begin{cases}
-(1-\mu)t, & \text{if } x \in [(1-\mu)t, \infty),   \\
-(1-\mu^2)t+\mu x & \text{if }  x \in [-\mu t , (1-\mu)t], \\
-t, & \text{if } x \in (-\infty, -\mu t]. 
\end{cases}
\end{align}
\begin{prop}
A viscosity solution of 
\begin{align}\label{eq:HJBc}
\begin{cases}
v_t+H_i(x,v_x)=0, \quad \text{ for } x \in \Omega_i, i=1,2, \\
\min \{v_t+H_1(x,v_x), v_t+H_2(x,v_x) \} \leq 0 \quad \text{ for }  x \in \mathcal{H}, \\
\max \{v_t+H_1(x,v_x), v_t+H_2(x,v_x) \} \geq 0 \quad \text{ for } x \in \mathcal{H}, \\
v(0,x)=0.   \tag{HJB}
\end{cases}
\end{align}
is given by \eqref{eq:sol}.
\end{prop}
\begin{proof}
The initial condition $v(0,x)=0$ is trivially satisfied. 
We show that $v$ is a subsolution. Suppose $\phi: [0,\infty) \times \mathbb{R} \to \mathbb{R}$ is differentiable, and $v-\phi$ achieves a local maximum $0$ at $(t_0,x_0) \in (0, \infty) \times \mathbb{R}$. Since $v$ is differentiable in the domain $O:=\{(t,x): t>0,\ x \not =(1-\mu)t,\ x \not=-\mu t \}$, we have $\phi_t(t_0,x_0)=v_t(t_0,x_0), \phi_x(t_0,x_0)=v_x(t_0,x_0)$ if $(t_0,x_0) \in O$. Then it is can be easily verified that $\phi_t+H_i(x,\phi_x)=0$ at $(t_0,x_0)$, where $i=1$ if $x_0 \geq 0$, and $i=2$ if $x_0 \leq 0$. 

Suppose $(t_0,x_0)$ is on the line $\{(t,x):t>0, \ x=(1-\mu)t\}$. Note that 
\begin{align*}
&\partial_t^- v(t_0,x_0)=-(1-\mu), \  \partial_t^+(t_0,x_0)=-(1-\mu^2), \\&\partial_x^- v(t_0,x_0)=\mu, \ \quad  \partial_x^+v(t_0,x_0)=0. 
\end{align*}
Since $(t_0,x_0)$ is a local maximum of $u-\phi$, we must have $$ (\phi_t(t_0,x_0), \phi_x(t_0,x_0)) \in \{ (r,p): r \in [ -(1-\mu^2),-(1-\mu)], \ p \in [0,\mu] \}.$$
Take $\Delta x =(1-\mu) \Delta t$. As a result of $$v(t_0+\Delta t, x_0+\Delta x)-\phi(t_0+\Delta t, x_0+\Delta x) \leq 0,$$ we obtain that 
\begin{align*}
-(1-\mu)\Delta t - \phi_t \Delta t -\phi_x \Delta x + \smallO({\Delta t}) \leq 0.
\end{align*}
Since we can choose $\Delta t$ to be either positive or negative, it can be easily deduced that $$-(1-\mu)-\phi_t-(1-\mu)\phi_x=0.$$
Substituting into $H_1$, we obtain that 
\begin{align*}
\phi_t(t_0,x_0)+H_1(x_0, \phi_x(t_0,x_0)) & = \phi_t(t_0,x_0)+(1-\mu)
+(1-\mu)\phi_x(t_0,x_0)=0.
\end{align*}

If $(t_0, x_0)$ is on the line $\{(t,x): t>0, \ x=-\mu t \}$, we have sub/super differentials of $v,$
\begin{align*}
&\partial_t^- v(t_0,x_0)=-1, \  \partial_t^+(t_0,x_0)=-(1-\mu^2), \\&\partial_x^- v(t_0,x_0)=0, \ \quad  \partial_x^+v(t_0,x_0)=\mu. 
\end{align*}
Therefore $v-\phi$ cannot achieve a local maximal on the line $\{(t,x): t>0, \ x=-\mu t \}$. Hence we have proved that $v$ is a subsolution of \eqref{eq:HJBc}, and similarly, we can show that  $v$ is a supersolution.
\end{proof}

\subsection{Control problem}
In this subsection, we show that there is a unique viscosity solution of \eqref{eq:HJBc} by applying results from \cite{MR3092359} and \cite{MR3206981}. First, we interpret \eqref{eq:HJBc} as a control problem. 

In the domain $\Omega_i, i=1,2$, we take $A_i=[0,1]$ as the space of controls, and $$b_i(x,\alpha_i)=\alpha_i \mu -(1-\alpha_i)(1-\mu), \quad  \alpha_i \in A_i,$$ as the controlled dynamics. For $x \in \mathcal{H}$, define the space of controls  $A:=A_1 \times A_2 \times [0,1]$, and the dynamics  $$b_{\mathcal{H}}(x, (\alpha_1,\alpha_2,c)):=c b_1(x,\alpha_1)+(1-c)b_2(x,\alpha_2), \quad (\alpha_1,\alpha_2,c) \in A.$$ The running cost in the domain $\Omega_1$ is given by $l_1(x,\alpha_1)=-(1-\mu)$,  in the domain $\Omega_2$ by $l_2(x,\alpha_2)=-\alpha_2$, and in $\mathcal{H}$ by $$l_{\mathcal{H}}(x,(\alpha_1,\alpha_2, c))=c l_1(x, \alpha_1)+(1-c)l_2(x, \alpha_2),$$
where $(\alpha_1,\alpha_2,c) \in A$.

In order to let trajectories stay on the boundary $\mathcal{H}$ for a while, for $x \in \mathcal{H}$, we denote 
\begin{align*}
A_0(x):=\{a=(\alpha_1,\alpha_2,c) \in A: \ b_{\mathcal{H}}(x, (\alpha_1,\alpha_2,c))=0  \}.
\end{align*}
We say a control $a \in A_0(x)$ is regular if $b_1(x,\alpha_1) \leq 0, b_2(x, \alpha_2) \geq 0$, and denote 
\begin{align*}
A_0^{reg}(x):=\{a=(\alpha_1,\alpha_2,c) \in A_0(x): (-1)^i b_i(x, \alpha_i) \geq 0 \}. 
\end{align*}

Define $\mathcal{A}:=L^{\infty}([0,1]; A)$. We say a Lipschitz function $X_x:[0,1] \to \mathbb{R}, X_x(0)=x$, an admissible trajectory if there exists some control process $a(\cdot) \in \mathcal{A}$, such that for a.e. $t \in [0,1]$
\begin{align}\label{eq:traj}
\dot{X}_{x}(t)=&b_1(X_{x}(t), \alpha_1(t)) \mathbbm{1}_{\{X_{x}(t) \in \Omega_1\}}
 +b_2(X_{x}(t), \alpha_2(t)) \mathbbm{1}_{\{X_{x}(t) \in \Omega_2\}}\\
 &+b_{\mathcal{H}}(X_{x}(t), (\alpha_1(t),\alpha_2(t), c(t)) \mathbbm{1}_{\{X_{x}(t) \in \mathcal{H}\}}. \notag
 \end{align}
According to \cite[Theorem 2.1]{MR3206981}, we have $a(t) \in A_0(X_x(t))$ for a.e. $t \in \{s: X_x(s) \in \mathcal{H}\}$. Denote by $\mathcal{T}_x$ the set of admissible controlled trajectories starting from $x$, i.e.,
\begin{align*}
\mathcal{T}_x:= \{(X_x(.), a(.)) \in \text{Lip}([0,1]; \mathbb{R}) \times \mathcal{A} \text{ such that \eqref{eq:traj} is satisfied and } X_x(0)=x      \}.
\end{align*}
Let us also introduce the set of regular trajectories, 
\begin{align*}
\mathcal{T}^{reg}_x:=\{(X_x(.), a(.)) \in \mathcal{T}_x: a(t) \in A_0^{reg}(X_x(t)) \text{ for a.e. } t \in \{s: X_x(s) \in \mathcal{H}\} \}
\end{align*}

For each $x \in \mathbb{R}, t \in [0,1)$, we define two value functions 
\begin{align}
V^-(x,t):=& \inf_{(X_x(.),a(.))\in \mathcal{T}_x }\int_0^t l(X_x(s),a(s) \ ds, \label{eq:Umin}\\
V^+(x,t):=& \inf_{(X_x(.),a(.))\in \mathcal{T}^{reg}_x }\int_0^t l(X_x(s),a(s) \ ds \label{eq:Umax},
\end{align}
where the cost function $l$ is given by
\begin{align*}
l(X_x(s),a(s)):=\sum_{i=1,2}l_i(X_x(s),\alpha_i(s))\mathbbm{1}_{\{X_x(s) \in \Omega_i\}}+ l_{\mathcal{H}}(X_x(s), a(s))\mathbbm{1}_{\{X_x(s) \in \mathcal{H}\}}.
\end{align*}

Note that in $\Omega_i, i=1,2$, the associated Hamiltonian of \eqref{eq:Umin} and \eqref{eq:Umax}
$$ (x, p) \mapsto \sup_{\alpha_i \in A_i} \{ -b_i(x,\alpha_i)p-l_i(x,\alpha_i)\}$$
coincides with $H_i$ in the last subsection. Then according to \cite[Theorem 3.3]{MR3206981}, both $V^-$ and $V^+$ are viscosity solutions of \eqref{eq:HJBc}. We will show that they are actually equal and there is only one viscosity solution of \eqref{eq:HJBc}. 
\begin{prop}\label{prop:viscosity}
$V^-=V^+$ is the unique viscosity solution of \eqref{eq:HJBc}, and $V^-$ is the minimal supersolution of \eqref{eq:HJBc}. 
\end{prop}
\begin{proof}
The argument is an application of results from \cite{MR3206981}. Define the Hamiltonians on $\mathcal{H}$ via 
\begin{align*}
H_T(x):=&\sup_{A_0(x)} \{ -l_{\mathcal{H}}(x,a) \}, \\
H_T^{reg}(x):=& \sup_{A^{reg}_0(x)} \{-l_{\mathcal{H}}(x,a) \},
\end{align*}
Let us compute $H_T(x)$. Suppose $a=(\alpha_1,\alpha_2,c) \in A_0(x)$. Then it can be easily verified that maximizing $-l_{\mathcal{H}}(x,a)$ over $A_0(x)$ is equivalent to maximizing  
\begin{align}\label{eq:max}
c(1-\mu)+(1-c)\alpha_2,
\end{align}
subject to constraints,
\begin{align}
&c(\alpha_1+\mu-1)+(1-c)(\alpha_2+\mu-1)=0,  \label{eq:constraint1} \\
& c, \alpha_1,\alpha_2  \in [0,1]. \notag
\end{align}

We first fix $\alpha_2$ and suppose $\alpha_2 > (1-\mu)$. Due to the equality $$c(1-\mu)+(1-c)\alpha_2=(1-\mu-\alpha_2)c+\alpha_2,$$ and the fact that the coefficient before $c$ is negative, maximizing \eqref{eq:max} is equivalent to minimizing $c$ under the constraints. It can be easily seen that the minimum $c$ can be obtained if and only if $\alpha_1=0$. Therefore the equation \eqref{eq:constraint1} becomes $1+\alpha_2c= \alpha_2 +\mu$, and hence \eqref{eq:max} is equal to $(1+c)(1-\mu)$. Now fix $\alpha_1=0$. In order to obtain the maximum of $c$, we have to take $\alpha_2=1$. In that case $\alpha_1=0,\alpha_2=1, c=\mu$ and $c(1-\mu)+(1-c)\alpha_2=1-\mu^2$. 

If $\alpha_2 \leq (1-\mu)$, we have $c(1-\mu)+(1-c)\alpha_2 \leq (1-\mu)< 1- \mu^2$. Since $(0,1,\mu)$ is a regular control, we conclude that $$H_T(x)=H_T^{reg}(x)=1-\mu^2.$$

We say a continuous function $v$ is viscosity solution of 
\begin{align}
v_t+\mathbb{H}^-(x, v_x)&=0 \text{ in } (0,1) \times \mathbb{R}, \label{eq:comparison}  \\
\big[ resp., \quad v_t+\mathbb{H}^+(x, v_x)&=0 \text{ in } (0,1) \times \mathbb{R}    \big] \notag
\end{align}
if it satisfies \eqref{eq:HJBc} and 
\begin{align*}
v_t+H_T(x)&=0 \text{ on } [0,1] \times \mathcal{H}, \\
[resp., \quad v_t+H_T^{reg}(x)&=0 \text{ on } [0,1]  \times \mathcal{H}].
\end{align*}
According to \cite[Theorem 3.3]{MR3206981}, $V^+$ is a viscosity subsolution of $v_t+\mathbb{H}^+(x,v_x)=0$, and hence also a viscosity subsolution of \eqref{eq:comparison} since $H_T=H_T^{reg}$ in our case. As a result of \cite[Theorem 4.2, 4.4]{MR3206981}, $V^-$ is the viscosity solution of \eqref{eq:comparison}, and the comparison result holds for \eqref{eq:comparison}. Therefore we conclude that $V^+ \leq V^-$. Then according to their definitions \eqref{eq:Umin} and \eqref{eq:Umax}, they must be equal. 

Finally according to \cite[Theorem 4.4]{MR3206981}, $V^-$ is the minimal supersolution of \eqref{eq:HJBc} and $V^+$ is the maximal subsolution of \eqref{eq:HJBc}. Then if $v$ is a viscosity solution of \eqref{eq:HJBc}, we must have $V^- \leq  v \leq V^+$ and hence $v=V^-=V^+$. 
\end{proof}

\subsection{Upper bound \eqref{eq:upper}}
In this subsection, we show that 
\begin{align*}
\underline{v}(t,x):=\liminf\limits_{(s,y,\delta) \to (t,x,0) } \tilde{V}^{\delta}(s,y)
\end{align*}
is a viscosity supersolution of \eqref{eq:HJBc}. Then according to Proposition~\ref{prop:viscosity}, we obtain that $\underline{v}(t,x) \geq v(t,x)$, and hence
\begin{align*}
\liminf\limits_{N \to \infty} \frac{\tilde{V}(N,Nx)}{N} \geq v(1,x).
\end{align*}
In particular, if we take $x=0$, then the above inequality becomes 
\begin{align}\label{eq:upper}
\limsup\limits_{N \to \infty} \frac{V(N, 1/2)}{N} \leq 1-\mu^2. 
\end{align}
\begin{prop}
\underline{v} is a viscosity supersolution of \eqref{eq:HJBc}. 
\end{prop}
\begin{proof}
The proof is almost the same as \cite[Theorem 2.1]{MR1115933}, and we record here for completeness. Fixing arbitrary $T>0$, we show that $\underline{v}$ is a viscosity supersolution over $[0,T] \times \mathbb{R}$.  Assume that $(t_0,x_0)$ is a strict local minimum of $\underline{v}-\phi$ for some $\phi \in \mathcal{C}_b^{\infty}([0,T]\times \mathbb{R})$. As a result of \eqref{eq:sdpp}, it can be easily seen that $\underline{v}(t,x) \in [-t,0]$. Without loss of generality, we assume that $t_0 \in (0,T), \underline{v}(t_0,x_0)=\phi(t_0,x_0)$, and there exists some $r>0$ such that 
\begin{enumerate}[(i)]
\item  $\phi \leq -2T$  outside the ball  $B((t_0,x_0),r):=\{(t,x): (t-t_0)^2+(x-x_0)^2 \leq r^2 \}$, 
\item $\underline{v}-\phi \geq 0= \underline(t_0,x_0)-\phi(t_0,x_0)$ in the ball $B((t_0,x_0),r)$. 
\end{enumerate}
Then there exists a sequence of $(t_n, x_n, \delta_n)$ such that $(t_n,x_n,\delta_n) \to (t_0,x_0,0)$ and $(t_n,x_n)$ is a global minimum of $\tilde{V}^{\delta_n}-\phi$. Due to the definition of $\underline{v}$, we have that $\xi_n:=\tilde{V}^{\delta_n}(t_n,x_n)-\phi(t_n,x_n) \to 0$ and $\tilde{V}^{\delta_n}(t,x) \leq \phi(t,x)+\xi_n$ for any $(t,x) \in [0,T] \times \mathbb{R}$.

According to \eqref{eq:sdpp}, we obtain that 
\begin{align}\label{eq:limit}
0 \leq    \phi(t_n,x_n)+& \max \left\{   \delta_n\left(1-\mu+\frac{\mu}{1+(1/\epsilon)^{x_n/\delta_n}} \right)  -\mu \phi(t_n-\delta_n, x_n+\delta_n) -(1-\mu) \phi(t_n-\delta_n,x_n), \right.\notag  \\
     &  \left.  \delta_n(1-\mu)\left(1-\frac{1}{1+(1/\epsilon)^{x_n/\delta_n}}\right)- (1-\mu) \phi\left(t_n -\delta_n, x_n-\delta_n ) \right)-\mu \phi(t_n-\delta_n,x_x) \right\}. 
\end{align} 
We prove for the case $x_0=0$, and the proof for $x \not =0$ is the same. Since $\left\{\frac{1}{1+(1/\epsilon)^{x_n/\delta_n}}\right\}_{n \geq 0} \in [0,1]$, we can take a convergent subsequence. For simplicity, we still denote it by $\left\{\frac{1}{1+(1/\epsilon)^{x_n/\delta_n}}\right\}_{n \geq 0}$, and assume it converges to some $s \in [0,1]$. Letting $n \to \infty$ in \eqref{eq:limit}, we obtain that 
\begin{align*}
0 \leq \phi_t(t_0,x_0)+ \max\left\{ 1-\mu+\mu s - \mu \phi_x(t_0,x_0), (1-\mu)(1-s)+(1-\mu)\phi_x(t_0,x_0)\right\}.
\end{align*}
Note that if $$1-\mu+\mu s - \mu \phi_x(t_0,x_0) \geq (1-\mu)(1-s)+(1-\mu)\phi_x(t_0,x_0),$$ then we have 
\begin{align*}
H_2(x_0, \phi_x(t_0,x_0)) \geq 1- \mu\phi_x(t_0,x_0) \geq 1-\mu+\mu s - \mu \phi_x(t_0,x_0),
\end{align*}
and hence $$\phi_t(t_0,x_0)+H_2(x_0, \phi_x(t_0,x_0) \geq 0.$$ Similarly if $$(1-\mu)(1-s)+(1-\mu)\phi_x(t_0,x_0) \geq 1-\mu+\mu s - \mu \phi_x(t_0,x_0),$$ then 
\begin{align*}
H_1(x_0, \phi_x(t_0,x_0))\geq 1-\mu+(1-\mu)\phi_x(t_0,x_0) \geq (1-\mu)(1-s)+(1-\mu)\phi_x(t_0,x_0),
\end{align*}
and hence $$\phi_t(t_0,x_0)+H_1(x_0, \phi_x(t_0,x_0) \geq 0.$$
Therefore, we have shown that $$\max\{\phi_t(t_0,x_0)+H_1(x_0,\phi_x(t_0,x_0)), \phi_t(t_0,x_0)+H_2(x_0,\phi_x(t_0,x_0)) \} \geq 0.$$
\end{proof}

\section{Lower Bound on the Value function}\label{section:lower}
It was proved in \cite{2020arXiv200100543R} that the asymptotic average value is $(1-\mu)$ for any offline strategy of the malicious expert if starting with weight $p_1^1=1/2$. Here we provide a lower bound on the value functions for the corresponding online problem 
\begin{align}\label{eq:lower}
\liminf\limits_{N \to \infty} \frac{V(N,\rho)}{N}\geq 1-\mu+\mu(1-\mu)(\rho-g(\rho))>1-\mu,
\end{align}
which shows that the malicious expert has more advantages when he adopts online policies.

This lower bound can be achieved if the malicious expert chooses to lie at state $\rho$ and chooses to tell the truth at state $g(\rho)$. For $p_1^1=\rho$, define the corresponding strategies by 
\begin{equation}
\alpha^{\rho}_t (\mathcal{G}_t)=
\begin{cases}
L  & \text{if } p_t^1= \rho, \\
T  & \text{if } p_t^1=g(\rho),
\end{cases}
\end{equation}
and $\alpha^{\rho}:=(\alpha_1^{\rho}, \alpha_2^{\rho}, \dotso)$. We denote the value function associated with $\alpha^{\rho}$ by 
\begin{align*}
V^{\alpha^{\rho}}(N,\rho)=\mathbb{E}^{\alpha^{\rho}}\left[\sum_{t=1}^N l(\hat{y}_t, y_t) \left| \ p_1^1=\rho \right. \right].
\end{align*}

\begin{prop}
\begin{align*}
\lim\limits_{N \to \infty} \frac{ V^{\alpha^{\rho}}(N,\rho)}{N}= 1-\mu+\mu(1-\mu)(\rho-g(\rho)).
\end{align*}
\end{prop}
\begin{proof}
Under strategy $\alpha^{\rho}$, $\{p^1_t\}_{t \in \mathbb{N}}$ is a Markov chain with two states $\{\rho, g(\rho)\}$ starting with $p_0^1=\rho$, and its transition probability is given by 
\begin{align*}
&\mathbb{P} \left[p_{t+1}^1=\rho \ | \ p_t^1=\rho\right]=1-\mu, \quad \quad  \mathbb{P} \left[p_{t+1}^1=g(\rho) \ | \ p_t^1=\rho\right]=\mu, \\
&\mathbb{P} \left[p_{t+1}^1=\rho \ | \ p_t^1=g(\rho)\right]=1-\mu, \quad  \mathbb{P} \left[p_{t+1}^1=g(\rho) \ | \ p_t^1=g(\rho)\right]=\mu.
\end{align*}
Denote its distribution at time $t$ by 
\begin{align*}
\pi_t:=\left( \mathbb{P}(p_t^1= \rho), \mathbb{P}(p_t^1=g(\rho))\right). 
\end{align*}
It can be easily seen that $(1-\mu, \mu)$ is the stationary distribution of $\{p_t^1\}_{t \in \mathbb{N}}$. According to \cite[Theorem 4.9]{MR3726904}, the distribution $\pi_t$ converges to $(1-\mu, \mu)$ as $t \to \infty$. 
Due to the equality
\begin{align*}
\mathbb{E}^{\alpha^{\rho}}\left[\sum_{t=0}^N l(\hat{y}_t, y_t) \left| \ p_0^1=\rho \right.\right]=\sum_{t=0}^N  \mathbb{P}(p_t^1=\rho)(1-\mu+\mu \rho) + \sum_{t=0}^N \mathbb{P}(p_t^1= g(\rho)) (1-\mu)(1-g(\rho)),
\end{align*}
it can be easily verified that 
\begin{align*}
 \lim\limits_{N \to \infty} \frac{ V^{\alpha^{\rho}}(N,\rho)}{N}&=(1-\mu)(1-\mu+\mu\rho)+\mu(1-\mu)(1-g(\rho)) \\
 &=1-\mu+\mu(1-\mu)(\rho-g(\rho)) \geq 1- \mu. 
\end{align*}

\end{proof}

\section{asymptotically optimal strategy for the forecaster}\label{section:adaptive}
In this section, we show that an adaptive multiplicative weightsed algorithm can resist corruptions of the malicious expert. Different from the multiplicative weights algorithm in Section~\ref{section:Problem}, the adaptive multiplicative weightsed algorithm updates the weights $p_t^i, i=1,2$, as follows: 
\begin{align*}
p_{t+1}^i= \frac{p_t^i e^{-\eta_t|x^i_t-y_t|}}{p_t^1 e^{-\eta_t|x^1_t-y_t|}+p_t^2 e^{-\eta_t|x^2_t-y_t|}},
\end{align*}
where $\eta_t =\sqrt{8 (\ln 2)/t }, t \in \mathbb{N}_+$ is time-varying. Denote by $V^*(N,\rho)$ the value function for the malicious expert under the adaptive multiplicative weightsed algorithm. Define 
$$g_t(p_t^1)=\frac{1}{1+(1/p_t^1-1)e^{\eta_t}}, \quad g_t^{(-1)}(p_t^1)=\frac{1}{1+(1/p_t^1-1)e^{-\eta_t}}.$$ 
It can be easily verified that $V^*(N,\rho)$ is the solution to dynamic programming equations
\begin{align*}
V^*(t+1,\rho) =& \max\{(1-\mu+\mu\rho)+\mu V^*(t, g_t(\rho))+(1-\mu)V^*(t,\rho),  \\
& (1-\mu)(1-\rho)+(1-\mu) V^*\left(t, g_t^{(-1)}(\rho)\right)+\mu V^*(t,\rho)       \}, \notag
\end{align*}
together with initial conditions $V^*(0,\rho)=0$. 
\begin{prop}\label{prop:asym}
\begin{align}
\lim\limits_{N \to \infty} \frac{V^*(N,1/2)}{N}= 1-\mu,
\end{align}
which implies that this adaptive multiplicative weights algorithm is asymptotically optimal for the forecaster. 
\end{prop}
\begin{proof}
Suppose the malicious expert keeps lying, i.e. taking strategies $\alpha_t(\mathcal{G}_t)=L, t \in \mathbb{N}_+$. Then according to \eqref{eq:currentloss}, it can be easily seen that the cumulative loss under this strategy is greater than or equal to $(1-\mu)N$, and hence
\begin{align*}
\liminf\limits_{N \to \infty} \frac{V^*(N,1/2)}{N} \geq 1-\mu.  
\end{align*}

To prove the other inequality, for any path $\mathcal{G}_{N+1}$ with $p_1^1=p_1^2 =1/2$, we define 
\begin{align*}
\hat{L}_N:= \sum_{t=1}^N l(\hat{y}_t, y_t), \quad  L^i_N:=\sum_{t=1}^N l(x^i_t, y_t), i=1,2. 
\end{align*}
Applying \cite[Chapter 2, Theorem 2.3]{MR2409394}, we obtain that 
\begin{align*}
\hat{L}_N- \min_{i=1,2} L^i_N \leq 2 \sqrt{\frac{N}{2}\ln 2}+\sqrt{\frac{\ln 2 }{8} },
\end{align*}
and hence
$$\hat{L}_N \leq  L^2_N +2 \sqrt{\frac{N}{2}\ln 2}+\sqrt{\frac{\ln 2 }{8} }.$$
Therefore for any strategy $\alpha$, we obtain
\begin{align*}
\mathbb{E}^{\alpha} \left[\hat{L}_N \left| p_1^1=1/2 \right. \right] &\leq \mathbb{E}^{\alpha} \left[L^2_N \left| p_1^1=1/2 \right. \right] +2 \sqrt{\frac{N}{2}\ln 2}+\sqrt{\frac{\ln 2 }{8} } \\
&= (1-\mu)N+2 \sqrt{\frac{N}{2}\ln 2}+\sqrt{\frac{\ln 2 }{8} },
\end{align*}
and also 
\begin{align*}
\limsup\limits_{N \to \infty} \frac{V^*(N,1/2)}{N} \leq 1-\mu. 
\end{align*}
\end{proof}

\section{Conclusions}\label{section:conclusions}
In this paper, we studied an online prediction problem with two experts of whom one is malicious. At each round, based on all the prior history, the malicious expert chooses to tell the true outcome or not so as to maximize the loss. We showed that the multiplicative weights algorithm cannot resist the corruption of the malicious expert by explicitly finding upper and lower bounds on the value function; see \eqref{eq:upper} and \eqref{eq:lower}. We also proved that an adaptive multiplicative weights algorithm can resist the corruption; see Proposition~\ref{prop:asym}. 
\bibliographystyle{siam}
\bibliography{ref}
\end{document}